\documentclass[12pt]{amsart}

 \usepackage[margin=1in]{geometry}

\newcommand{\private}[1]{}
\usepackage{amssymb, amsmath, amscd, amsthm, color, epsfig,url, graphicx, url}

\usepackage[all]{xy}          % for xy-pic pictures
\xyoption{dvips}              % for xy-pic pictures

\makeatletter
\renewcommand\l@subsection{\@tocline{2}{0pt}{2pc}{5pc}{}}
\makeatother

\newcommand{\R}{{\mathbb R}}

\newcommand{\Emb}{\operatorname{Emb}}

\newcommand{\Imm}{\operatorname{Imm}}

\newcommand{\rImm}{\operatorname{rImm}}
\newcommand{\rImmce}{\rImm^{\mathrm{ce}}\!}
\newcommand{\CEmb}{\operatorname{CEmb}}
\newcommand{\smallrImm}{\rImm^{\mathrm{sm}}\!}
\newcommand{\affImm}{\rImm^{\mathrm{aff}}\!}
\newcommand{\rbar}{\overline{\rImm}}

\newcommand{\Conf}{\operatorname{Conf}}
\newcommand{\rConf}{\operatorname{rConf}}

\newcommand{\linear}{\operatorname{L}}
\newcommand{\inj}{\operatorname{Linj}}
\newcommand{\ev}{\operatorname{ev}}
\newcommand{\SafeDistance}{sd}
\newcommand{\radius}{R}
\newcommand{\lr}{LR}

\theoremstyle{plain}
\newtheorem{thm}{Theorem}[section]
\newtheorem{prop}[thm]{Proposition}
\newtheorem{lemma}[thm]{Lemma}
\newtheorem{cor}[thm]{Corollary}

\theoremstyle{definition}
\newtheorem{defin}[thm]{Definition}

\newtheorem{def/ex}[thm]{Definition/Example}

\theoremstyle{remark}
\newtheorem{rem}[thm]{Remark}

\usepackage{tikz}

\makeatletter
\@namedef{subjclassname@2020}{%
  \textup{2020} Mathematics Subject Classification}
\makeatother

\usepackage{graphicx}

\begin{document}
\pagestyle{plain}
\title{The space of $r$-immersions of a union of discs in $\mathbb R^n$}

\author{Gregory Arone}
\address{Department of Mathematics, Stockholm University}
\email{gregory.arone@math.su.se}
%\urladdr{???}

\author{Franjo \v Sar\v cevi\'c}
\address{Department of Mathematics and Computer Science, University of Sarajevo}
\email{franjo.sarcevic@pmf.unsa.ba}
\urladdr{pmf.unsa.ba/franjos}

\thanks{\textbf{Acknowledgements.} F. \v Sar\v cevi\'c was partially supported by the grant P20\_01109 (JUNTA/FEDER, UE). The authors thank the anonymous referee for his/her careful reading and helpful comments that improved the quality of the manuscript.} 

\begin{abstract}
For a manifold $M$ and an integer $r>1$, the space of $r$-immersions of $M$ in $\R^n$ is defined to be the space of immersions of $M$ in $\R^n$ such that the preimage of every point in $\R^n$ contains fewer than $r$ points. We consider the space of $r$-immersions when $M$ is a disjoint union of $k$ $m$-dimensional discs, and prove that it is equivalent to the product of the $r$-configuration space of $k$ points in $\R^n$ and the $k^{\text{th}}$ power of the space of injective linear maps from $\R^m$ to $\R^n$. This result is needed in order to apply Michael Weiss's manifold calculus to the study of $r$-immersions. The analogous statement for spaces of embeddings is ``well-known'', but a detailed proof is hard to find in the literature, and the existing proofs seem to use the isotopy extension theorem, if only as a matter of convenience. Isotopy extension does not hold for $r$-immersions, so we spell out the details of a proof that avoids using it, and applies to spaces of $r$-immersions.
\end{abstract}

\maketitle

%%%%%%%%%%%%%%%%%%%%%%%%%%%%%%%%%%%%%%%%%%%%%%

%\tableofcontents

%%%%%%%%%%%%%%%%%%%%%%%%%%%%%%%%%%%%%%%%%%%%%%%%%%%%%%%%%%%%%%%%%%%%%%%%%%%%%%%%%%%%%%%%%%%%%%%%%%%%%%%%

\parskip=5pt
\parindent=0cm

%%%%%%%%%%%%%%%%%%%%%%%%%%%%%%%%%%%%%%%%%%%%%%%%%%%%%%%%%%%%%%%%%%%%%%%%%%%%%%%%%

%\section{Introduction}\label{S:Intro}

%%%%%%%%%%%%%%%%%%%%%%%%%%%%%%%%%%%%%%%%%%%%%%%%%%%%%%%%%%%%%%%%%%%%%%%%%%%%%%%%%

\begin{section}{Introduction}
Embedding calculus (also known as manifold calculus) is a method, invented by M. Weiss~\cite{W:EI1}, for analysing presheaves on manifolds. Suppose $F$ is a contravariant functor defined on a suitable category of $m$-dimensional manifolds. Let $D^m$ be the open unit disc in $\R^m$. One of the main ideas of embedding calculus is to first focus on the value of $F$ on manifolds of the form $\coprod_{i=1}^k D^m$, and then extrapolate from there to get approximations to the value of $F$ on general $m$-dimensional manifolds. For this approach to be useful, one generally needs a good understanding of the value of $F$ on disjoint unions of copies of $D^m$. 

The original motivating example for embedding calculus is, not coincidentally, the embedding functor $\Emb(-, \R^n)$, were $\R^n$ is a fixed vector space and the domain of the functor is considered to be the category of $m$-dimensional manifolds and codimension zero embeddings, for some fixed $m$. One can also replace $\R^n$ with a more general manifold $N$, but we will restrict ourselves to embeddings into a Euclidean space.

In order to apply Weiss's machinery to the embedding functor, one needs a good understanding of the homotopy type of spaces of the form $\Emb(\coprod_k D^m, \R^n)$. Fortunately, the homotopy type of these spaces is well-understood. To describe it, let $\Conf(k, \R^n) \subset (\R^n)^k$ be the configuration space of ordered $k$-tuples of pairwise distinct points in $\R^n$. That is, for $\underline{k}=\{1,\ldots,k\}$ $$\Conf(k, \R^n):=\Emb(\underline{k},\R^n)\cong \{(x_1,\ldots,x_k)\in (\R^n)^k: x_i \neq x_j \text{ for } i \neq j\}.$$
Also let $\inj(\R^m, \R^n)$ denote the space of injective linear maps from $\R^m$ to $\R^n$.

There is a natural map 
\begin{equation}\label{eq:EmbConf}
\Emb \left(\coprod_k D^m, \R^n\right) \xrightarrow{\simeq} \Conf(k, \R^n)\times \inj(\R^m, \R^n)^k.
\end{equation}
The map is defined by evaluating an embedding at the centers of the discs, and also differentiating  at the centers of the discs. It is well-known that the map~\eqref{eq:EmbConf} is an equivalence.

Recently there has been an emerging interest in applying embedding calculus to the study of {\it $r$-immersions}~\cite{DT:Homology, SSV:r-imm}. Given an integer $r>1$, an $r$-immersion is an immersion with the property that the preimage of each point consists of fewer than $r$ points. Let $\rImm(M, N)$ denote the space of $r$-immersions of $M$ into $N$. For $r=2$, a $2$-immersion is the same thing as an injective immersion. When $M$ is compact it is the same as an embedding. More generally, when $M$ is tame (i.e., is the interior of a compact manifold with boundary), the space of $2$-immersions is equivalent to the space of embeddings. Thus for practical purposes, we can identify the space of $2$-immersions with the space of embeddings. In general, we have inclusions 
\[\Emb(M, N) \subset 3\Imm(M, N) \subset \cdots \subset\rImm(M, N) \subset \cdots \subset \Imm(M, N).\]

In order to apply embedding calculus to the study of $r$-immersions, one would like to have an analogue of the equivalence~\eqref{eq:EmbConf} for $r$-immersions. Let $\rConf(k, \R^n)$, called the \textit{$r$-configuration space}, also known as  \textit{non $r$-equal configuration space}, of $k$ points in $\R^n$, be defined to be the space
$$\rConf(k,\R^n):=\rImm(\underline{k},\R^n)\cong \{(x_1,\ldots,x_k) \in (\R^n)^k : \nexists 1\le i_1<\cdots < i_r \le k \text{ s.t. } x_{i_1}=\ldots=x_{i_r} \}.$$ 

There is a natural map
\begin{equation}\label{rImmrConf}
\rImm\left(\coprod_k D^m, \R^n\right) \xrightarrow{\simeq} \rConf(k, \R^n)\times \inj(\R^m, \R^n)^k,
\end{equation}
defined by evaluation at the centers of the discs and differentiation at the centers of the discs. It is a generalization of~\eqref{eq:EmbConf}. Our main result (Theorem~\ref{theorem: rimm to config}) says that this map is an equivalence. This fact is implicitly assumed in~\cite{SSV:r-imm}, and to some extent also in~\cite{DT:Homology}.

As we mentioned above, the case $r=2$ of our result is well-known. However, we had trouble finding a detailed proof of it in the literature. Furthermore, proofs that we did find tend to use at some point the fact that when $M_0$ is a closed submanifold of $M$, the restriction maps $\Emb(M, \R^n) \to \Emb(M_0, \R^n)$ and $\Imm(M, \R^n) \to \Imm(M_0, \R^n)$ are fibrations. This property definitely fails for $r$-immersions when $2<r<\infty$. Even when $M=\underline{k}$ is a finite set (a zero-dimensional manifold), and $\underline{k_0}\subset \underline{k}$, the restriction map $\rConf(k, \R^n)\to \rConf(k_0, \R^n)$ is not a fibration.

Our companion paper \cite{AS:homological} relies heavily on the map~\eqref{rImmrConf} being an equivalence. It thus seems prudent to write out a full proof of this assertion. We hope that the result is interesting enough to stand on its own.

%However, while the proofs in the literature of the fact (\ref{EmbConf}) that we know of use the isotopy extension theorem (at least for simplifying the argument), \textbf{ADD SOME REFERENCES HERE}\footnote{\textbf{Franjo:} Which references do you think are worth citing here?} the argument cannot be identical in case of $r$-immersions. Naimely, isotopy extension theorem does not hold for $r$-immersions, because the restriction maps for $r$-immersions are not fibrations even for $r$-immersions of zero-dimensional manifolds which are $r$-configuration spaces, so it is not obvious that the usual proof for $r=2$ extends to the case of general $r$. Therefore, it seems prudent to establish the relationship between the space of $r$-immersions of a union of discs and the $r$-configuration space.
%In this paper, we prove that there is a weak homotopy equivalence (\ref{rImmrConf}).

%With that equivalence in mind, %we can consider $\rImm^{(k)}(\emptyset, \R^n)$ as the total homotopy fiber of the $k$-cubical diagram 
%$$S \mapsto \rConf(k-|S|,\R^n)$$ with projection maps, so 
%the Taylor tower for the space $\rImm(M,\R^n)$ to a large part comes down to the cubes of partial configuration spaces.

%We also believe that the result is interesting in itself.

\subsubsection*{Notation for the derivative of a function} Since the letter $D$ denotes a disc in $\R^m$, we avoid using it to denote the differential. Instead we use the same notation for derivatives as in~\cite{C:calculus}. Let $U, V$ be open subsets of $\R^m$ and $\R^n$ respectively, and let $f\colon U \to V$ be a smooth function. We define $f'$ to be the function $f'\colon U \to \linear(\R^m, \R^n)$ which associates to a point $x\in U$ the Fr\'echet derivative of $f$ at $x$. Thus for each $x\in U, f'(x)$ is a linear homomorphism from $\R^m$ to $\R^n$. The linear approximation of $f$ at $0$ can be written as $Lf(x)= f(0)+f'(0)(x)$. 

Similarly $f''\colon U \to \linear\left((\R^m\otimes \R^m)_{\Sigma_2}, \R^n\right)$ denotes the second derivative of $f$, and $f^{(i)}\colon U \to \linear\left((\R^{m})^{\otimes i}_{\Sigma_i}, \R^n\right)$ denotes the $i$-th derivative of $f$. Thus for each $x\in U$, $f^{(i)}(x)$ is a symmetric multilinear map from $\R^m\times \cdots \times \R^m$ to $\R^n$. 

We will equip the space of multilinear maps with the norm metric. Thus $$||f^{(i)}(x)||_{\mathrm{norm}}=\sup\{||f^{(i)}(x)(u_1, \ldots, u_i)||_{\R^n} : ||u_1||_{\R^m}= \cdots =||u_i||_{\R^m}=1\}.$$
Subsequently, we will just write $||-||$ to denote the norm of a vector or an operator, trusting that it is clear from the context which norm is meant.

This notion of derivative can be extended to smooth maps $f\colon M\to N$ where $M, N$ are smooth manifolds with boundary. We will only apply it to the case when $N=\R^n$ and $M$ is a finite disjoint union of either open or closed discs in $\R^m$.

\subsubsection*{Topology on spaces of smooth maps} We endow the space $C^\infty(M, \R^n)$ of smooth maps from $M$ to $\R^n$ with the compact-open $C^\infty$-topology, a.k.a the weak Whitney $C^\infty$-topology. Thus a sequence of smooth functions $f_j\colon M\to \R^n$ converges to $f$ if for every $k\ge 0$ the sequence $f_j^{(k)}$ converges to $f^{(k)}$ uniformly on every compact subspace of $M$. When $M$ is a subset of $\R^m$, this is equivalent to saying that for every multi-index $\alpha$, the sequence $\partial^\alpha f_j$ converges to $\partial^\alpha f$ uniformly on every compact subspace of $M$. All spaces of embeddings, immersions, $r$-immersions, etc. are endowed with the subspace topology from the space of smooth maps.
\end{section}

\begin{section}{$r$-immersions of a union of discs}\label{S:rIC}

%In this section we establish the relationship between the space of $r$-immersions of a union of discs and the $r$-configuration space. %This section owes a debt to Grossnickle's thesis~\cite[Section 2.3]{G:Thesis}. The main result of this section can be viewed as a strengthening of Proposition 2.3.1 in~[loc.~cit].
Let $D^m$ be the open unit disc in $\R^m$ and $k$ a natural number. Let $\inj(\R^m, \R^n)$ denote the space of injective linear maps from $\R^m$ to $\R^n$. There is a canonical map
\[
\ev\colon \rImm\left(\coprod_k D^m, \R^n\right) \to \rConf(k, \R^n)\times \inj(\R^m, \R^n)^k.
\]
The map $\ev$ is defined by evaluating an $r$-immersion $f\colon \coprod_k D^m \to \R^n$ at the center of each disc $D^m$ and also differentiating $f$ at the center of each disc. The main result of this paper is the following theorem
\begin{thm}\label{theorem: rimm to config}
The map $\ev$ is a weak homotopy equivalence.
\end{thm}
%When $r=2$, or when we replace $\rImm$ and $\rConf(k, \R^n)$ with $\Emb$ and $\Conf(k, \R^n)$ respectively, the theorem is well-known. The proofs in the literature \cite{MV:Cubes}, \cite{M:MfldCalc} of this fact that we know of use the isotopy extension theorem. Since isotopy extension does not hold for $r$-immersions \cite, it is not obvious to us that the usual proof for the case $r=2$ extends to the case of general $r$. So it seems prudent to include a proof here.

The proof of Theorem~\ref{theorem: rimm to config} goes through a construction of several intermediate spaces. To begin with, let us introduce notation for some auxiliary spaces of componentwise embeddings.
\begin{defin}
Suppose $\coprod_{i=1}^k M_i$ is a disjoint union of $k$ manifolds, possibly with boundary. Let $$\CEmb\left(\coprod_{i=1}^k M_i, \R^n\right)\subset C^\infty\left(\coprod_{i=1}^k M_i, \R^n\right)$$ be the space of smooth maps that restrict to an embedding on each component. Similarly, let $\rImmce\left(\coprod_{i=1}^k M_i, \R^n\right)$ be the subspace of $\rImm\left(\coprod_{i=1}^k M_i, \R^n\right)$ consisting of those $r$-immersions that restrict to an embedding on each $M_i$. \end{defin}
In practice we will use this notation only when each $M_i$ is an open or closed disc in $\R^m$.
\begin{rem}\label{remark:open}
It is well-known that for a compact $M$, possibly with boundary,  $\Emb(M, N)$ is open in $C^\infty(M, N)$ (since $M$ is compact, the strong and weak Whitney topologies coincide). See~\cite[Theorem 1.4]{Hirsch:Differential} or~\cite[Proposition 9.5.9]{MR-OD:DiffTop}, where the case with boundary is treated more explicitly. Since there are homeomorphisms $\CEmb\left(\coprod_{i=1}^k M_i, \R^n\right)\xrightarrow{\cong} \prod_i \Emb\left(M_i, \R^n\right)$ and $C^\infty\left(\coprod_{i=1}^k M_i, \R^n\right)\xrightarrow{\cong} \prod_i C^\infty\left(M_i, \R^n\right)$, it follows that when the manifolds $M_i$ are compact, $\CEmb\left(\coprod_{i=1}^k M_i, \R^n\right)$ is an open subset of $C^\infty\left(\coprod_{i=1}^k M_i, \R^n\right)$. Since $$\rImmce\left(\coprod_{i=1}^k M_i, \R^n\right)= \rImm\left(\coprod_{i=1}^k M_i, \R^n\right) \cap \CEmb\left(\coprod_{i=1}^k M_i, \R^n\right),$$ it follows that, again assuming $M_i$ are compact, $\rImmce\left(\coprod_{i=1}^k M_i, \R^n\right)$ is an open subset of $\rImm\left(\coprod_{i=1}^k M_i, \R^n\right)$. 
\end{rem}
We will now define a map that ``almost'' exhibits $\rImmce\left(\coprod_k D^m, \R^n\right)$ as a deformation retract of 
$\rImm\left(\coprod_k D^m, \R^n\right)$. 
\begin{defin} \label{def:almost retraction}
Let $f \in C^\infty\left(\coprod_k D^m, \R^n\right)$ and $0\le t \le 1$. Define the map $f_t\in C^\infty\left(\coprod_k D^m, \R^n\right)$ by the formula $f_t(x)=f(tx)$. It is clear that if $f$ is an $r$-immersion then so is $f_t$ for all $0<t\le 1$, but not for $t=0$. Likewise, if $f$ is a componentwise embedding, then so is $f_t$ for all $0<t\le 1$. Define the map 
\[
H\colon \rImm\left(\coprod_k D^m, \R^n\right) \times (0, 1] \to \rImm\left(\coprod_k D^m, \R^n\right)
\]
by the formula
\[
H(f, t)=f_t.
\]
\end{defin}
%Clearly, $H$ restricts to a map 
%\[
%\rImmce\left(\coprod_k D^m, \R^n\right) \times (0, 1] \to \rImmce\left(\coprod_k D^m, \R^n\right).
%\] 
For all $f\in \rImm\left(\coprod_k D^m, \R^n\right)$, since $f$ is a local embedding, there exists an $\epsilon >0$ such that for all $0 <t\le \epsilon$, $H(f, t)\in \rImmce\left(\coprod_k D^m, \R^n\right)$. The following easy lemma is a strengthening of this obsevation.
\begin{lemma}\label{lemma: local retract}
For every $f\in \rImm\left(\coprod_k D^m, \R^n\right)$ there exists an open neighbourhood $U$ of $f$ and an $\epsilon>0$ such that $H(U\times (0, \epsilon])\subset \rImmce\left(\coprod_k D^m, \R^n\right)$.
\end{lemma}
\begin{proof}
Let $f\colon \coprod_k D^m \to \R^n$ be an $r$-immersion. For $0<\epsilon <1$, let $B^m_\epsilon\subset D^m$ denote the closed disc of radius $\epsilon$. Since $f$ is a local embedding one can find an $\epsilon>0$ so that $f$ restricts to an embedding of $B^m_\epsilon$ for each copy of $D^m$. This means that the restriction map
\[
\rho\colon \rImm\left(\coprod_k D^m, \R^n\right) \rightarrow \rImm\left(\coprod_k B^m_\epsilon, \R^n\right)
\]
takes $f$ into the subspace $\rImmce\left(\coprod_k B^m_\epsilon, \R^n\right)\subset  \rImm\left(\coprod_k B^m_\epsilon, \R^n\right)$. By Remark~\ref{remark:open} it is an open subset. Thus $\rho^{-1}\left(\rImmce\left(\coprod_k B^m_\epsilon, \R^n\right)\right)$ is the required open neighborhood of $f$ in $\rImm\left(\coprod_k D^m, \R^n\right)$.
\end{proof}
%The following corollary follows from Lemma~\ref{lemma: local retract} by a routine compactness argument.%\footnote{\textbf{Franjo:} What kind of compactness argument do you have in your mind? It is intuitive for me, but I would like to make that completely clear to me.}
\begin{cor}\label{cor: retract for compact}
For any compact space $K$ and map $K\to \rImm\left(\coprod_k D^m, \R^n\right)$ there exists an $\epsilon>0$ such that the composition
\[
K\times (0, 1] \to \rImm\left(\coprod_k D^m, \R^n\right)\times (0, 1] \xrightarrow{H} \rImm\left(\coprod_k D^m, \R^n\right)
\]
takes $K\times (0, \epsilon]$ into $ \rImmce\left(\coprod_k D^m, \R^n\right)$.
\end{cor}
\begin{proof}
Suppose we have a map $h: K \to \rImm(\coprod_k D^m, \R^n)$. Let $x \in K$. Then $h(x)$ is an $r$-immersion. By Lemma \ref{lemma: local retract}, there is an open neighborhood $V$ of $h(x)$ in $\rImm(\coprod_k D^m, \R^n)$, and a positive number $\epsilon_x$ such that all the elements of $H(V \times (0,\epsilon_{x}])$ are componentwise embeddings. Let $U_x=h^{-1}(V)$. Thus for every $x \in K$ we found an open neighborhood $U_x$ and a positive number $\epsilon_x$ such that every $r$-immersion in $h(U_x)$ restricts to an embedding of discs of radius $\epsilon_x$. By compactness of $K$, there is a finite collection of points, say $x_1, …, x_l$ such that $U_{x_1}, …, U_{x_l}$ cover $K$. Let $\epsilon=\min\{\epsilon_{x_1}, …, \epsilon_{x_l}\}$. Then every $r$-immersion in  $h(K)$ restricts to an embedding of discs of radius $\epsilon$. But this means that every element of $H( h(K) \times (0, \epsilon])$ is a componentwise embedding.
\end{proof}
Now we can complete the first important step toward proving Theorem~\ref{theorem: rimm to config}.
\begin{prop}
The inclusion 
\begin{equation}\label{eq:inclusion} 
\rImmce\left(\coprod_k D^m, \R^n\right) \hookrightarrow \rImm\left(\coprod_k D^m, \R^n\right)
\end{equation}
is a weak homotopy equivalence.
\end{prop}
\begin{proof}
First of all, the inclusion is surjective on $\pi_0$. Indeed, suppose $f\in \rImm\left(\coprod_k D^m, \R^n\right)$. Then $H(\{f\}\times (0, 1])$ defines a path from $f$ to a point in $\rImmce\left(\coprod_k D^m, \R^n\right)$.

Second, let us show that the inclusion is injective on $\pi_0$. Suppose $f, g\in \rImmce\left(\coprod_k D^m, \R^n\right)$ and there is a path $\alpha\colon [0, 1]\to \rImm\left(\coprod_k D^m, \R^n\right)$ from $f$ to $g$. Consider the composition %$H\circ (\alpha\times 1_{(0, 1]})$
\[
[0, 1]\times (0, 1] \xrightarrow{\alpha\times 1_{(0, 1]}} \rImm\left(\coprod_k D^m, \R^n\right)\times (0, 1]\xrightarrow{H} \rImm\left(\coprod_k D^m, \R^n\right).
\]
It follows from Corollary~\ref{cor: retract for compact} that for some $0<\epsilon < 1$ this map restricts to a map
\[
[0, 1]\times [\epsilon, 1] \xrightarrow{H\circ(\alpha\times 1_{[\epsilon, 1]})} \rImm\left(\coprod_k D^m, \R^n\right)
\]
that sends $[0, 1]\times \{\epsilon\}$ into $\rImmce\left(\coprod_k D^m, \R^n\right)$. Moreover, since $\alpha$ maps $\partial([0, 1])$ into the subspace $\rImmce\left(\coprod_k D^m, \R^n\right)$, and $H$ preserves this subspace, it follows that $H\circ(\alpha\times 1_{[\epsilon, 1]})$ maps $\partial([0, 1])\times [\epsilon, 1]$ into $\rImmce\left(\coprod_k D^m, \R^n\right)$. Altogether it follows that  $H\circ(\alpha\times 1_{[\epsilon, 1]})$ takes $[0, 1]\times \{\epsilon\}\cup \partial([0, 1])\times [\epsilon, 1]$ into $\rImmce\left(\coprod_k D^m, \R^n\right)$. Therefore $H\circ(\alpha\times 1_{[\epsilon, 1]})$ defines a path homotopy between $\alpha$, and a path from $f$ to $g$ that lies entirely in $\rImmce\left(\coprod_k D^m, \R^n\right)$. We have proved that the inclusion is injective on $\pi_0$.

Now let us choose a basepoint in $\rImmce\left(\coprod_k D^m, \R^n\right)$ and let it also serve as the basepoint of $\rImm\left(\coprod_k D^m, \R^n\right)$. We want to show that for all $d\ge 1$ the induced homomorphism
\[
\pi_d\left(\rImmce\left(\coprod_k D^m, \R^n\right)\right)\to \pi_d\left(\rImm\left(\coprod_k D^m, \R^n\right)\right)
\]
is an isomorphism.
For surjectivity, let $h\colon S^d\to \rImm\left(\coprod_k D^m, \R^n\right)$ be a pointed map. We need to show that $h$ is pointed homotopic to a map $S^d \to \rImmce\left(\coprod_k D^m, \R^n\right)$. Using Corollary~\ref{cor: retract for compact} once more we conclude that there exists an $0< \epsilon <1$ such that the composition
\[
S^d\times [\epsilon, 1] \xrightarrow{h\times 1_{[\epsilon, 1]}} \rImm\left(\coprod_k D^m, \R^n\right) \times [\epsilon, 1] \xrightarrow{H} \rImm\left(\coprod_k D^m, \R^n\right)
\]
takes $S^d\times\{\epsilon\}$
%\footnote{\textbf{Franjo:} Why $\times\{\epsilon\}$ Ok!}
into $\rImmce\left(\coprod_k D^m, \R^n\right)$. 
We obtained an {\it unpointed} homotopy of $h$ to a map $h_\epsilon\colon S^d\to \rImmce\left(\coprod_k D^m, \R^n\right)$. Furthermore, $h$ takes the basepoint of $S^d$ into $\rImmce\left(\coprod_k D^m, \R^n\right)$, so it follows that the homotopy $H\circ (h\times 1_{[\epsilon, 1]})$, while not constant on the basepoint, keeps the basepoint inside $\rImmce\left(\coprod_k D^m, \R^n\right)$. It follows that $h$ is pointed homotopic to a conjugation of $h_\epsilon$ by a path in $\rImmce\left(\coprod_k D^m, \R^n\right)$, which completes the proof of surjectivity on $\pi_d$.

Finally, we need to show that the inclusion is injective on $\pi_d$. Suppose $h\colon S^d\to \rImmce\left(\coprod_k D^m, \R^n\right)$ represents an element of the kernel. It means that $h$ extends to a map $\tilde h\colon D^{d+1}\to \rImm\left(\coprod_k D^m, \R^n\right)$. Using $H$ and Corollary~\ref{cor: retract for compact} once again, one can show that $\tilde h$ can be deformed into a map $D^{d+1} \to \rImmce\left(\coprod_k D^m, \R^n\right)$ that defines a null homotopy of $h$, thus proving that $h$ represents zero in $\pi_d\left(\rImmce\left(\coprod_k D^m, \R^n\right)\right)$. The details of this last step are left to the reader. 
\end{proof}
Let us say that an immersion of $\coprod_k D^m$ is non-$r$-overlapping if the intersection of images of every $r$ components is empty. Note that a componentwise embedding is an $r$-immersion if and only if it is non-$r$-overlapping.

The following definition is taken from \cite{G:Thesis}. 
\begin{defin}
Let $f\colon \coprod_k D^m \to \R^n$ be a map. For each $1\le i \le k$, let $0_i$ be the center of the $i$-th copy of $D^m$ in the coproduct $\coprod_k D^m$. For each $r$-tuple of integers $\vec{i}=(i_1, \ldots, i_r)$, where $1\le i_1<i_2<\cdots <i_r\le k$, let 
$$
f_{\vec{i}}=\frac{f(0_{i_1}) +\cdots + f(0_{i_r})}{r}.
$$
Finally define $\SafeDistance(f)$ by the following formula
\[
\SafeDistance(f)=\frac{1}{\sqrt{r}} \ \underset{1\le i_1<\cdots <i_r\le k}{\min}\sqrt{\sum_{j=1}^r || f(0_{i_j}) - f_{\vec{i}}||^2}.
\]%\footnote{\textbf{Franjo:} Should stay here $f(0_j)$ or $f(0_{i_j})$?}
\end{defin}
The notation $\SafeDistance$ stands for ``safe distance''. It is a distance for which it is guaranteed that if the radius of each disc is less than the safe distance, then the immersion is non-$r$-overlapping (as we will prove shortly). The formula for $\SafeDistance(f)$ does not give the largest possible safe distance, but what matters is that $\SafeDistance(f)$ depends continuously on $f$.

Next, let us make precise the notion of a radius of an immersion of a union of discs.
\begin{defin}
Let $f\colon \coprod_k D^m \to \R^n$ be a map. Let $D^m_i$ denote the $i$-th copy of $D^m$ in the coproduct. Define the radius of $f$ to be the following
\[
\radius(f)=\sup_{1\le i \le k, \, x\in D^m_i} ||f(x)-f(0_i)||.
\]
\end{defin}
Note that $\radius(f)$ can be $\infty$. But if, for example, $f$ is the restriction of a map defined on a union of closed unit discs then $\radius(f)$ is finite.

The point of the last two definitions is that they give a condition for an immersion to be non-$r$-overlapping. The following lemma is present implicitly in~\cite{G:Thesis}.
\begin{lemma}\label{lemma:nonoverlapping}
Let $f\colon \coprod_k D^m \to \R^n$ be a map. If $\radius(f)<\SafeDistance(f)$ then $f$ is non-$r$-overlapping.
\end{lemma}
\begin{proof}
Suppose by contradiction that $f$ is $r$-overlapping. Then there exists an $r$-tuple $\vec{i}=(i_1,\ldots, i_r)$ and points $x_{i_j}\in D^m_{i_j}$ such that $f(x_{i_1})=\cdots = f(x_{i_r})$. Let $z$ denote this common value. Recall that $f_{\vec{i}}$ is the centroid of $f(0_{i_1}), \ldots, f(0_{i_r})$. We have the following inequalities
\[
r\cdot\SafeDistance(f)^2\le \sum_{j=1}^r ||f(0_{i_j})-f_{\vec{i}}||^2 \le \sum_{j=1}^r ||f(0_{i_j})-z||^2\le r\cdot\radius(f)^2
\]%\footnote{Franjo: I changed $f(0_j)$ to $f(0_{i_j})$. Right?}
which contradicts the assumption $\radius(f)<\SafeDistance(f)$.
\end{proof}
Next, let us introduce the homotopy between a smooth map $f\colon\coprod_k D^m \to \R^n$ and its linearization. It is a standard tool in the study of embeddings of a disc or a union of discs.
\begin{defin}\label{def:linearization homotopy}
Define the map $\Phi\colon C^\infty(\coprod_k D^m, \R^n)\times [0, 1] \to C^\infty (\coprod_k D^m, \R^n)$ as follows. Let $D^m_i$ denote the $i$-th copy of $D^m$ in $\coprod_k D^m$, let $0_i$ be the center of $D^m_i$ and let $x\in D^m_i$. Then 
\[
\Phi(f, t)(x)=\left\{\begin{array}{cc}  f(0_i)+ \frac{f(xt)-f(0_i)}{t} & t>0 \\[5pt] f(0_i)+f'(0_i)(x) & t=0 \end{array}\right.
\]
\end{defin}
We need to know that $\Phi$ is a continuous function. This is a standard result, but we did not find a detailed proof of it, so for the reader's convenience we include one.
\begin{lemma}\label{lem:linearization homotopy continuous}
The function $\Phi$ of Definition~\ref{def:linearization homotopy} is continuous.
\end{lemma}
\begin{proof}
Continuity at points where $t>0$ really is obvious and is left to the reader. We shall address continuity at points where $t=0$. The weak topology on $C^\infty(\coprod_k D^m, \R^n)$ is first countable, so it is enough to prove that $\Phi$ is sequentially continuous. Suppose we have a sequence $(f_j, t_j)$ in $C^\infty(\coprod_k D^m, \R^n)\times [0, 1]$ converging to $(f, 0)\in C^\infty(\coprod_k D^m, \R^n)\times [0, 1]$. We have to show that $\Phi(f_j, t_j)$ converges to the linearization of $f$. This means that we have to show that $\Phi(f_j, t_j)$ converges uniformly to $f$ on any compact subset of $\coprod_k D^m$, and the same holds for all derivatives of these functions.

Let us first prove convergence on the level of functions themselves. We can use linear Taylor approximation to write, for each $j$ and $x\in D^m_i$
\[
f_j(x)=f_j(0_i)+f_j'(0_i)(x)+E_j(x)
\]
where $E_j$ is the error term. It follows that 
\begin{equation}\label{eq:Philinearization}
\Phi(f_j, t_j)(x)=f_j(0_i)+f_j'(0_i)(x) + \frac{E_j(xt_j)}{t_j}
\end{equation}
where by convention $\frac{E_j(xt)}{t}=0$ when $t=0$. 

Let $K\subset \coprod_k D^m$ be a compact subset. We need to show that $\Phi(f_j, t_j)(x)$ converges to $\Phi(f, 0)(x)$ uniformly in $x$, when $x$ is restricted to $K$. Let us define the constants $M_j, M$ as follows
\[
M_j=\sup_{x\in K}(||f_j''(x)||), \quad M=\sup_{x\in K}(||f''(x)||)
\]
where $||-||$ denotes the operator norm. Since $K$ is compact, $M_j, M$ are finite. Since $f_j$ converges to $f$ in the $C^\infty$ topology, the sequence $M_j$ converges to $M$, and in particular it is bounded. 

By Taylor's theorem for vector-valued functions~\cite[Theorem 5.6.2]{C:calculus} we have the estimate
\[
||E_j(xt)||\le \frac{M_j}{2}||x||^2 t^2 \le\frac{M_j}{2} t^2.
\]
Note that $E_j(xt)\in \R^n$ and $x\in \R^m$, so the two occurrences of $||-||$ denote the Euclidean norm in $\R^n$ and $\R^m$ respectively. It follows that the error estimate $\frac{E_j(xt_j)}{t_j}$ in~\eqref{eq:Philinearization} satisfies the following estimate for all $j$ and $x$:
\begin{equation}\label{eq:estimate}
\left\|\frac{E_j(xt_j)}{t_j}\right\|\le \frac{M_j}{2} t_j.
\end{equation}
Therefore the following holds, where as usual $x\in D^m_i$:
\[
\begin{array}{cl}
\left\|\Phi(f, 0)(x)-\Phi(f_j, t_j)(x)\right\|  & =\left\|f(0_i)+f'(0_i)(x)-\left(f_j(0_i)+f_j'(0_i)(x)+\frac{E_j(xt_j)}{t_j}\right)\right\|\ \\[5pt] & \le \left\|f(0_i)-f_j(0_i)\right\| + \left\|f'(0_i)-f_j'(0_i)\right\| + \frac{M_j}{2} t_j.
\end{array}
\]
Here the $||-||$ sign refers to the euclidean norm in $\R^n$ in $\left\|f(0_i)-f_j(0_i)\right\|$, and to the operator norm in  $ \left\|f'(0_i)-f_j'(0_i)\right\|$. Since $f_j$ converges to $f$ in the weak Whitney  $C^\infty$-topology, and $t_j\xrightarrow{n\to \infty} 0$, and $M_j$ is bounded, it is clear that the right hand side of the inequality converges to zero as $j\to \infty$, independently of $x$. We have proved the convergence on the level of functions.

Now let us look at derivatives. Fix a compact set $K$ as above. It is easy to check that for all $j$ and $x\in D^m_i$, $\Phi(f_j, t_j)'(x)=f_j'(x t_j)$, while $\Phi(f, 0)'(x)=f'(0_i)$. We have an estimate
\[
\left\|f_j'(xt_j)-f_j'(0_i)\right\| \le M_j t_j.
\]
Since $t_j$ converges to $0$, it follows that by taking $j$ large enough we can make $\left\|f_j'(xt_j)-f_j'(0_i)\right\|$ arbitrarily small, for all $x\in K\cap D^m_i$. Since $f_j$ converges to $f$ in the $C^\infty$-topology, we can also make $||f_j'(0_i)-f'(0_i)||$ arbitrarily small. It follows that by taking $j$ large enough we can make $||f_j'(xt_j)-f'(0_j)||=\left\|\Phi(f_j, t_j)'(x)- \Phi(f, 0)'(x)\right\|$ arbitrarily small, which means that $\Phi(f_j, t_j)'$ converges to $\Phi(f, 0)'$ uniformly on $K$.

Finally, suppose $i>1$. It is not hard to check that $\Phi(f_j, t_j)^{(i)}(x)=f_j^{(i)}(x t_j)t_j^{i-1}$ and $\Phi(f, 0)^{(i)}(x)=0$. Since $f_j^{(i)}$ converges to $f^{(i)}$ uniformly on $K$, it follows that $||f_j^{(i)}(x)||$ is uniformly bounded on $K$, and thus $||f_j^{(i)}(x t_j)t_j^{i-1}||$ can be made arbitrarily small by taking $j$ large enough. This proves convergence for higher derivatives.
\end{proof}
We need to define one more invariant of a map of a union of discs. It measures the largest possible radius attained by $f$ during the homotopy $\Phi$. 
\begin{defin}
Let $f\colon \coprod_k D^m \to \R^n$ be a map. Let $D^m_i$ denote the $i$-th copy of $D^m$ in the coproduct. Define $\lr(f)$ by the following formula
\[
\lr(f)=\sup_{1\le i \le k, \, 0\le t\le 1, \, x\in D^m_i}||\Phi(f, t)(x)-f(0_i)||= \sup_{1\le i \le k, \, 0<t\le 1, \, x\in D^m_i} \frac{||f(tx)-f(0_i)||}{t}.
\]
\end{defin}
Just as with $\radius(f)$, $\lr(f)$ can be infinite. Indeed, it always holds that $\radius(f)\le \lr(f)$. But if, for example, $f$ is the restriction of a differentiable function defined on a union of closed discs then $\lr(f)<\infty$. Recall that for $0\le s\le 1$, $f_s$ is defined by the formula $f_s(x)=f(sx)$. It follows that if $f$ is differentiable and $0\le s<1$ then $\lr(f_s)<\infty$. Let us record this simple fact in a lemma.
\begin{lemma}\label{lemma: properties}
Suppose $f$ is differentiable 
\begin{enumerate}
\item Whenever $s<1$, $\lr(f_s)<\infty$. \label{finite}
\item $\lr(f_s)\le s\cdot \lr(f)$.\label{sublinear}
\end{enumerate}
\end{lemma}
\begin{proof}
(1) The case $s=0$ is trivial, because $f_0$ is constant on each disc, and $\lr(f_0)=0$. For $x \in D_{i}^m$ and $0<s<1$, $sx \in sD_{i}^m$, where $sD_{i}^m$ is the open disc of radius $s$ contained in $D_{i}^m$ whose closure $\overline{sD_{i}^m}$ is also contained in $D_{i}^m$, so it is easy to define an extension of $f_s$ to a union of closed discs.

(2) The case $s=0$ is trivial again. Let $s > 0$. For $x\in D_{i}^m$, $sx \in sD_{i}^m$, and if $0_i$ is the center of the disc $D_{i}^m$, then $0_i$ is also the center of the disc $sD_{i}^m$. By definition, 
\begin{equation*}
\begin{aligned}
\lr(f_s)=\sup_{1\le i \le k, \, 0<t\le 1, \, sx\in sD^m_i} \frac{||f(stx)-f(0_i)||}{t}
& =s\cdot \sup_{1\le i \le k, \, 0<t\le 1, \, sx\in sD^m_i} \frac{||f(stx)-f(0_i)||}{st} \\
& \leq s \cdot \sup_{1\le i \le k, \, 0<t\le 1, \, x\in D^m_i} \frac{||f(tx)-f(0_i)||}{t} \\
& = s\cdot \lr(f)
\end{aligned}
\end{equation*}%\footnote{Franjo: though... not so trivial, $\frac{1}{st} > \frac{1}{t}$ It is ok, among all t from 0 to 1 are also st's}
\end{proof}
We have the following simple but important observation.
\begin{lemma}\label{lemma: continuous}
Fix an $0<s<1$. Then $\lr(f_s)$ depends continuously on $f$.
\end{lemma}
\begin{proof}
For a fixed $f\in C^\infty(\coprod_k D^m, \R^n)$, let us define the function $\Psi_f\colon (\coprod_k D^m)\times [0, 1]\to \R^n$ as follows: for $x\in D^m_i$ and $0\le t \le 1$
\[
\Psi_f(x, t)=\Phi(f, t)(x) - f(0_i)=\left\{\begin{array}{cc} \frac{f(xt)-f(0_i)}{t} & t>0 \\ f'(0_i)(x) & t=0 \end{array}\right.
\]
Suppose we have a sequence $f_j\in C^\infty(\coprod_k D^m, \R^n)$, converging to $f$ in the usual $C^\infty$-topology, and we fix an $s<1$. We claim that in this case $\Psi_{(f_j)_s}$ converges uniformly to $\Psi_{f_s}$. Since $\lr(f)=\sup(\Psi_f)$, it follows that the sequence $\lr((f_j)_s)$ converges to $\lr(f_s)$, which is what we want to prove.

It remains to prove the claim. Since the sequence $f_j$ converges to $f$, and $s<1$, it follows that the sequence $(f_j)_s$ converges uniformly to $f_s$ on the entire space $\coprod_k D^m$, and same holds for all derivatives. It follows easily that for any $\delta>0$, $\Psi_{(f_j)_s}$ converges uniformly to $\Psi_{f_s}$ on the space $\left(\coprod_k D^m\right) \times [\delta, 1]$. To establish convergence near $t=0$, we use the estimates in the proof of Lemma~\ref{lem:linearization homotopy continuous}. Applying formula~\eqref{eq:Philinearization}, we can write the following:
\[
\Psi_{(f_j)_s}(x, t)-\Psi_{f_s}(x, t) =(f_j)_s'(0_i)(x) + \frac{E_j(xt)}{t} - \left(f_s'(0_i)(x) + \frac{E(xt)}{t}\right).
\]
Here $E$ is the error term for the linear approximation of $f$. Defining $M_j$ and $M$ as in the proof of Lemma~\ref{lem:linearization homotopy continuous}, and using inequality~\eqref{eq:estimate}, we get the following inequalities
\[
\begin{array}{cl}
\left\|\Psi_{(f_j)_s}(x, t)-\Psi_{f_s}(x, t)\right\| & \le \left\| (f_j)_s'(0_i)(x) - f_s'(0_i)(x)\right\| + \left\|\frac{E_j(xt)}{t}\right\| + \left\|\frac{E(xt)}{t}\right\| \\[5pt]
& \le \left\|(f_j)_s'(0_i) - f_s'(0_i)\right\| + \frac{M_j}{2}t + \frac{M}{2}t.
\end{array}
\]
Note that in the last line $\left\|(f_j)_s'(0_i) - f_s'(0_i)\right\|$ denotes the operator norm of $(f_j)_s'(0_i) - f_s'(0_i)$. 

%where $E_j(xt)$ is the error term. Let  us define
%\[
%M_j=\sup_{x\in \coprod_k D^m}(||(f_j)_s''(x)||)
%\]
%where $||-||$ denotes the operator norm. Note that since $s<1$, $(f_{n})_s$ is the restriction of a smooth function defined on a compact set, which implies that $M_j$ is finite. By Taylor's theorem for vector-valued functions~\cite[Theorem 5.6.2]{C:calculus} we have the estimate
%\[
%||E_j(xt)||\le \frac{M_j}{2}||x||^2 t^2 \le\frac{M_j}{2} t^2.
%\]
%Note that $E_j(xt)\in \R^n$ and $x\in R^m$, so the two occurrences of $||-||$ denote the Euclidean norm in $\R^n$ and $\R^m$ respectively. It follows that for all $x\in D^m_i$
%\[
%g_j(x, t)= (f_j)_s'(0_i)(x) + \frac{E_j(xt)}{t}
%\]
%where the error term $\frac{E_j(xt)}{t}$ satisfies the following inequality for all $0\le t \le 1$
%\begin{equation}\label{eq: estimate}
%\left\|\frac{E_j(xt)}{t}\right\|\le \frac{M_j}{2} t.
%\end{equation}
%There is a similar estimate for $g(x, t)$ with remainder term $\left\|\frac{E(xt)}{t}\right\|$, and constant $M$, determined by the second derivative of $f_s$. 
Since $(f_j)_s'(0_i)$ converges to $f_s'(0_i)$, and the sequence $M_j$ is bounded, it is clear that for all $\epsilon>0$ we can find a $\delta>0$ and an integer $j_1$, such that for all $0\le t\le \delta$ and $j>j_1$ the terms 
$||(f_j)_s'(0_i) - f_s'(0_i)||$, $\frac{M_j}{2}t$ and $\frac{M}{2}t$ are each smaller than $\frac{\epsilon}{3}$. It follows that for all $j>j_1$ and $(x, t)\in \coprod_k D^m \times [0, \delta]$, $\left\|\Psi_{(f_j)_s}(x, t)-\Psi_{f_s}(x, t)\right\|<\epsilon$.

We also can find an $j_2$ such that for all $j>j_2$ and $(x, t)\in \coprod_k D^m \times [\delta, 1]$ it holds
$
||\Psi_{(f_j)_s}(x, t)-\Psi_{f_s}(x, t)|| <\epsilon
$. Thus for all $j>\max(j_1, j_2)$ the inequality $
||\Psi_{(f_j)_s}(x, t)-\Psi_{f_s}(x, t)|| <\epsilon
$ holds for all $x$ and $t$. This means that $\Psi_{(f_j)_s}$ converges uniformly to $\Psi_{f_s}$, and we have proved the claim.
\end{proof}

%Suppose $f$ is, as usual, a function $\coprod_k D^m \to \R^n$, and $0 \le s \le 1$. Define the function $f_s$ by the formula $f_s(x)=f(sx)$. %The proof of the following elementary lemma is left to the reader\footnote{\textbf{Franjo:} Maybe it would be better to write it down?}

\begin{defin}
We say that a function $f\colon \coprod_k D^m \to \R^n$ is small, if $\lr(f)<\SafeDistance(f)$. Let 
\[
\smallrImm\left(\coprod_k D^m, \R^n\right)\subset \rImmce\left(\coprod_k D^m, \R^n\right)
\]
denote the subspace consisting of $r$-immersions that are componentwise embeddings and are small.
\end{defin}
Note that if $f$ is small then $\Phi(f, t)$ satisfies the hypothesis of Lemma~\ref{lemma:nonoverlapping} for all $t$, and thus $\Phi(f, t)$ is non-$r$-overlapping for all $t$.
\begin{prop}
The inclusion $\smallrImm\left(\coprod_k D^m, \R^n\right)\hookrightarrow \rImmce\left(\coprod_k D^m, \R^n\right)$ is a homotopy equivalence.
\end{prop}
\begin{proof}
Let us define the function $\alpha\colon \rImmce\left(\coprod_k D^m, \R^n\right)  \to (0, 1)$ by the formula
\[
\alpha(f)=\min\left(\frac{1}{2}, \frac{\SafeDistance(f)}{4\lr(f_{\!\frac{1}{2}})}\right).
\]
Notice that $\lr(f_{\!\frac{1}{2}})<\infty$ by Lemma~\ref{lemma: properties} \eqref{finite}, and therefore $\alpha(f)$ is a well-defined positive number smaller than $1$. Also notice that $\SafeDistance(f)$ is obviously continuous in $f$, and $\lr(f_{\!\frac{1}{2}})$ is continuous by Lemma~\ref{lemma: continuous}. Therefore $\alpha$ is a continuous function.

Next, let $j_f\colon [0, 1]\to [\alpha(f), 1]$ be the canonical linear homeomorphism. Let $j$ be the function
\[
j\colon \rImmce\left(\coprod_k D^m, \R^n\right) \times [0, 1] \to [0, 1]
\]
defined by the formula $j(f, t)=j_f(t)$, then $j$ is continuous in both $f$ and $t$.

Now let us define a homotopy 
\[
H\colon \rImmce\left(\coprod_k D^m, \R^n\right) \times [0, 1] \to \rImmce\left(\coprod_k D^m, \R^n\right)
\]
by the formula $H(f, t)=f_{j_f(t)}$. Since $j_f$ is continuous in $f$ and $t$, $H$ is continuous. Since $j_f(1)=1$ and $j_f(0)=\alpha(f)$, $H$ is a homotopy between the identity function on $\rImmce\left(\coprod_k D^m, \R^n\right)$ and the function that sends $f$ to $f_{\alpha(f)}=f_{\min(\frac{1}{2}, \frac{\SafeDistance(f)}{4\lr(f_{\!\frac{1}{2}})})}$.

Let us check that $f_{\min(\frac{1}{2}, \frac{\SafeDistance(f)}{4\lr(f_{\!\frac{1}{2}})})}$ is small. This means to check that 
\[
\lr(f_{\min(\frac{1}{2}, \frac{\SafeDistance(f)}{4\lr(f_{\!\frac{1}{2}})})})\le \SafeDistance(f_{\min(\frac{1}{2}, \frac{\SafeDistance(f)}{4\lr(f_{\!\frac{1}{2}})})}).
\] 
Note that $\SafeDistance(f)$ only depends on the images of the centers of $D^m$s under $f$, and therefore $\SafeDistance(f_s)=\SafeDistance(f)$ for any $s$. So we need to prove that $\lr(f_{\min(\frac{1}{2}, \frac{\SafeDistance(f)}{4\lr(f_{\!\frac{1}{2}})})})\le \SafeDistance(f)$.

Suppose first that $\frac{\SafeDistance(f)}{4\lr(f_{\!\frac{1}{2}})}\le \frac{1}{2}$. Then we have the inequalities (here we use Lemma~\ref{lemma: properties} \eqref{sublinear})
\[
\lr(f_{\min(\frac{1}{2}, \frac{\SafeDistance(f)}{4\lr(f_{\!\frac{1}{2}})})})=\lr(f_{\frac{\SafeDistance(f)}{4\lr(f_{\!\frac{1}{2}})}})\le \frac{\SafeDistance(f)}{2\lr(f_{\!\frac{1}{2}})}\lr(f_{\frac{1}{2}})=\frac{\SafeDistance(f)}{2}<\SafeDistance(f).
\]
Now suppose that $\frac{1}{2}\le \frac{\SafeDistance(f)}{4\lr(f_{\!\frac{1}{2}})}$. Then we have the inequality
$\lr(f_{\!\frac{1}{2}})\le \frac{\SafeDistance(f)}{2}$ and
\[
\lr(f_{\min(\frac{1}{2}, \frac{\SafeDistance(f)}{4\lr(f_{\!\frac{1}{2}})})})=\lr(f_{\frac{1}{2}})\le \frac{\SafeDistance(f)}{2}<\SafeDistance(f).
\]
We have shown that $H(f, 0)$ is small for every $f$. It is clear that if $f$ is small, then $H(f, t)$ is small for all $t$. We have shown that $H$ induces a homotopy between the identity map on $\rImmce\left(\coprod_k D^m, \R^n\right)$ and a map $\rImmce\left(\coprod_k D^m, \R^n\right)\to \smallrImm\left(\coprod_k D^m, \R^n\right)$, which serves as a homotopy inverse to the inclusion.
\end{proof}
The next step is to show that the space of small $r$-immersions that are componentwise embeddings is equivalent to the space of small $r$-immersions that are componentwise affine.
\begin{defin}
Let $\affImm\left(\coprod_k D^m, \R^n\right) \subset \smallrImm\left(\coprod_k D^m, \R^n\right)$ be the subspace consisting of $r$-immersions that are affine on each component (and are small).
\end{defin}
\begin{prop}\label{prop: small to linear}
The space $\affImm\left(\coprod_k D^m, \R^n\right)$ is a deformation retract of $\smallrImm\left(\coprod_k D^m, \R^n\right)$.
\end{prop}
\begin{proof}
%Let us define the map $H\colon \smallrImm\left(\coprod_k D^m, \R^n\right) \times [0, 1] \to \smallrImm\left(\coprod_k D^m, \R^n\right)$ by the following formula. Suppose $D^m_i$ is the $i$-th copy of $D^m$ in the coproduct, and $x\in D^m_i$. Then
%\[
%H(f, t)(x)=\left\{\begin{array}{ll} f(0_i)+\frac{f(tx)-f(0_i)}{t}, & t>0 \\ \hspace{0.24cm}f(0_i)+f'(0_i)(x), & t=0 \end{array}\right.
%\]
%It follows from Lemma~\ref{lem:linearization homotopy continuous} that the function $H$ is continuous. 
Recall the map $\Phi\colon C^\infty(\coprod_k D^m, \R^n)\times [0, 1] \to C^\infty (\coprod_k D^m, \R^n)$ from Definition~\ref{def:linearization homotopy}. 
It is easy to check from the definitions that  
\begin{enumerate}
\item $\Phi$ restricts to a map $\smallrImm\left(\coprod_k D^m, \R^n\right) \times [0, 1] \to \smallrImm\left(\coprod_k D^m, \R^n\right)$
%For every fixed $t$ and $f\in \smallrImm\left(\coprod_k D^m, \R^n\right)$, the function $x\mapsto \Phi(f, t)(x)$ is an element of $\smallrImm\left(\coprod_k D^m, \R^n\right)$.
\item If $f$ is affine on each component, meaning that $f(x)=f(0_i)+f'(0_i)(x)$ for all $x\in D_i^m$, then it's easily seen that $\Phi(f,t)(x)=f(x)$ for all $t$. So the homotopy $\Phi$ is constant on $\affImm\left(\coprod_k D^m, \R^n\right)$.
\item For every $f$ the function $x\mapsto \Phi(f, 0)(x)$ is affine on each component.
\end{enumerate}
It follows that $\Phi$ defines a deformation retraction of $ \smallrImm\left(\coprod_k D^m, \R^n\right)$ onto $\affImm\left(\coprod_k D^m, \R^n\right)$.
\end{proof}
Next, we can prove that the evaluation map restricted to $\affImm\left(\coprod_k D^m, \R^n\right)$ is an equivalence.
\begin{lemma}
The map 
\[
\ev \colon \affImm\left(\coprod_k D^m, \R^n\right)\to 
\rConf(k, \R^n)\times \inj(\R^m, \R^n)^k
\]
defined by evaluating at the center of each disc $D^m$ and also differentiating at the center of each disc, is a homotopy equivalence.
\end{lemma}
\begin{proof}
%A linear map is determined by its derivative and a single value, so it almost looks as if the map is a homeomorphism. But remember that we defined $\affImm\left(\coprod_k D^m, \R^n\right)$ to be a subspace of $\smallrImm\left(\coprod_k D^m, \R^n\right)$. I.e., it is the space of componentwise linear $r$-immersions that are also small. 
For affine maps, the smallness condition amounts to the following inequality, that has to hold for each $i$ between $1$ and $k$:
\[
||f'(0_i)||\le \SafeDistance(f).
\]
Here $||f'(0_i)||$ denotes the operator norm of $f'(0_i)$. It is easy to show that the space $\rConf(k, \R^n)\times \inj(\R^m, \R^n)^k$ deformation retracts onto the image of small affine $r$-immersions. One has to multiply the linear transformations from $\R^m$ to $\R^n$ by a factor that will make their norms smaller than $\SafeDistance(f)$.
\end{proof}
Finally we can prove the main result.
\begin{proof}[Proof of Theorem~\ref{theorem: rimm to config}]
We have constructed the following composition of maps
\begin{multline*}
\affImm\left(\coprod_k D^m, \R^n\right) \hookrightarrow \smallrImm\left(\coprod_k D^m, \R^n\right) \hookrightarrow \rImmce\left(\coprod_k D^m, \R^n\right)\hookrightarrow \\ \hookrightarrow \rImm\left(\coprod_k D^m, \R^n\right) \xrightarrow{\ev} \rConf(k, \R^n)\times \inj(\R^m, \R^n)^k.
\end{multline*}
We have shown that each one of the inclusions is an equivalence, and that the composition is an equivalence. It follows that the map marked $\ev$ is an equivalence.
\end{proof}
\begin{cor}
Choose a basepoint in $\Imm\left(\coprod_k D^m, \R^n\right)$, and let $\rbar\left(\coprod_k D^m, \R^n\right)$ be the homotopy fiber of the map $\rImm\left(\coprod_k D^m, \R^n\right)\to \Imm\left(\coprod_k D^m, \R^n\right)$. Then there exists an equivalence
%the following composition
%\[
%\rbar\left(\coprod_k D^m, \R^n\right) \to \rImm\left(\coprod_k D^m, \R^n\right) \to \rConf(k, \R^n)
%\]
%is an equivalence. %\footnote{\textbf{Franjo:} Maybe add: "i.e. $\rbar\left(\coprod_k D^m, \R^n\right) \simeq \rConf(k, \R^n)$"}
 %i. e. 
$$\rbar\left(\coprod_k D^m, \R^n\right) \simeq \rConf(k, \R^n).$$
\end{cor}
%\footnote{\textbf{Franjo:} In the case it will stay as a separate paper, what could we add as a further comments section? The case of general $N$ instead of $\R^n$?}
\end{section}

\bibliographystyle{alpha}

%\bibliography{Bibliography}

\pagestyle{plain}

\end{document}